\documentclass[letterpaper,12pt]{article}
\usepackage{amsmath,amssymb,amsfonts,epsf,epsfig,enumerate}

\newtheorem{theo}{Theorem}[section]

\newtheorem{lemma}[theo]{Lemma}
\newtheorem{defn}[theo]{Definition}

\newtheorem{cor}[theo]{Corollary}

\newenvironment{proof}{\noindent {\sc Proof}.}
                {\phantom{a} \hfill \framebox[2.2mm]{ } \bigskip}

\newcommand{\ZZ}{\mathbb{Z}}

\def\int{{\rm int}}

\newcommand{\F}{{\mathcal{F}}}
\newcommand{\T}{{\mathcal{T}}}
\newcommand{\C}{{\mathcal{C}}}
\newcommand{\G}{{\mathcal{G}}}
\renewcommand{\phi}{\varphi}
\newcommand\mset[1]{\left\{\!\!\left\{#1\right\}\!\!\right\}}  % multisets

\title{Spanning Euler tours and spanning Euler families \\ in hypergraphs with particular vertex cuts}
\author{Mateja \v{S}ajna and Yan D. Steimle\footnote{Main author. Email: ystei087@uottawa.ca. Mailing address: Department of Mathematics and Statistics, University of Ottawa, 585 King Edward Avenue, Ottawa, ON, K1N 6N5,Canada.} \\
{\small University of Ottawa}}

\begin{document}
\maketitle \baselineskip 17pt

\begin{abstract}
An {\em Euler tour} in a hypergraph is a closed walk that traverses each edge of the hypergraph exactly once, while an {\em Euler family}, first defined by Bahmanian and \v{S}ajna, is a family of closed walks that jointly traverse each edge exactly once and cannot be concatenated. In this paper, we study the notions of a {\em spanning Euler tour} and a {\em spanning Euler family}, that is, an Euler tour (family) that also traverses each vertex of the hypergraph at least once. We examine necessary and sufficient conditions for a hypergraph to admit a spanning Euler family, most notably, when the hypergraph possesses a vertex cut consisting of vertices of degree two. Moreover, we characterise hypergraphs with a vertex cut of cardinality at most two that admit a spanning Euler tour (family). This result enables us to reduce the problem of existence of a spanning Euler tour (which is NP-complete), as well as the problem of a spanning Euler family, to smaller hypergraphs.

\end{abstract}

\section{Introduction}

One of the best known and most accessible results in graph theory, Euler's Theorem, states that a connected graph admits an Euler tour --- that is, a closed walk traversing each edge of the graph exactly once --- if and only if all vertices of the graph have even degree. In addition to the most obvious way to generalise the notion of an Euler tour to hypergraphs, which has been studied in \cite{lonc-naroski,eulerianHG,Quasi-eulerianHG}, Bahmanian and \v{S}ajna \cite{eulerianHG,Quasi-eulerianHG} also introduced the notion of an {\em Euler family}, which is a family of closed walks that jointly traverse each edge of the hypergraph exactly once and cannot be concatenated. For connected graphs, the notions of an Euler tour and Euler family coincide; for general connected hypergraphs, however, they give rise to two rather distinct problems, the former NP-complete and the latter of polynomial complexity \cite{lonc-naroski, Quasi-eulerianHG}.

In this paper, we investigate eulerian substructures that are {\em spanning}; that is, in addition to traversing each edge exactly once, they also traverse each vertex of the hypergraph at least once. In a connected graph, every Euler tour is spanning; in a general connected hypergraph, however, not every Euler tour or family is spanning.

This paper is organised as follows. After an overview of basic hypergraph terminology in Section~\ref{sec:prelim}, we present in Section~\ref{sec:general}  some basic necessary conditions for a hypergraph to admit a spanning Euler family, as well as a characterisation of such hypergraphs via their incidence graphs. In Sections~\ref{sec:vx-cuts}--\ref{sec:2-vx-cuts}, we then focus on the impact of particular vertex cuts on the existence of a spanning Euler tour (family). In our first main result, Theorem~\ref{the:deg2}, we show that a hypergraph $H$ with a minimal vertex cut  consisting of vertices of degree two admits a spanning Euler family if and only if some of its derived hypergraphs (hypergraphs closely related to particular subhypergraphs of $H$)  admit spanning Euler families. Moreover, in Theorems~\ref{the:cutvx} and \ref{the:even}--\ref{the:oddSET}, we show that a hypergraph with a vertex cut  of cardinality at most two admits a spanning Euler family (tour) if and only if some of its derived hypergraphs admit spanning Euler families (tours). Hence, when studying the problem of existence of a spanning Euler family or tour, it suffices to consider connected hypergraphs without such vertex cuts, thereby reducing the problem.

\section{Preliminaries}\label{sec:prelim}

We begin with some basic concepts related to hypergraphs, which will be used in later discussions. For any graph- and hypergraph-theoretic terms not defined here, we refer the reader to \cite{GT} and \cite{HGconnection}, respectively.

A {\em hypergraph} $H$ is an ordered pair $(V,E)$, where $V$ is a non-empty finite set and $E$ is a finite multiset of $2^V$. (To denote multisets, we shall use double braces, $\mset{.}$.) The elements of $V = V(H)$ and $E = E(H)$ are called  {\em vertices} and  {\em edges}, respectively. A hypergraph is said to be {\em empty} if it has no edges.

Let $H = (V,E)$ be a hypergraph, and $u,v\in V$. If $u \ne v$ and there exists an edge $e\in E$ such that $u,v\in e$, then we say that $u$ and $v$ are {\em adjacent (via the edge $e$)}, and that $u$ is a {\em neighbour} of $v$  in $H$. The set of all neighbours of $v$ in $H$ is called the {\em neighbourhood} of $v$ in $H$, and is denoted by $N_H(v)$.
Two distinct edges $e,f\in E$ are {\em adjacent} in $H$ if $e\cap f \ne \emptyset$. If $v\in V$ and  $e\in E$ are such that $v\in e$, then $v$ is said to be {\em incident} with $e$, and the ordered pair $(v,e)$ is called a {\em flag} of $H$. The set of flags of $H$ is denoted by $F(H)$.
The {\em degree} of a vertex $v\in V$ is the number of edges in $E$ incident with $v$, and is denoted by $\deg_H(v)$, or simply $\deg(v)$ when there is no ambiguity. A vertex of degree 1 is called {\em pendant}. %For $V'\subseteq V$, the hypergraph $H$ is said to be {\em $r$-regular on} $V'$ if all vertices in $V'$ have degree $r$, and simply {\em $r$-regular} if it is $r$-regular on $V$. Furthermore, $H$ is said to be {\em even (odd) on} $V'$ if all vertices in $V'$ have even (respectively, odd) degree in $H$, and {\em even (odd)} if it is even (respectively, odd) on $V$.

The {\em incidence graph} of a hypergraph $H = (V,E)$ is the graph $\mathcal{G}(H) = (V_G,E_G)$ where $V_G = V\cup E$ and $E_G = \lbrace ve: (v,e) \in F(H) \rbrace$. Hence, $\mathcal{G}(H)$ is simple with bipartition $\lbrace V,E\rbrace$, and $E_G$ can be identified with the flag set $F(H)$. Furthermore, we call $x\in V_G$ a {\em v-vertex} if $x\in V$, and an {\em e-vertex} if $x\in E$.

A hypergraph $H' = (V',E')$ is called a {\em subhypergraph} of the hypergraph $H = (V,E)$ if $V'\subseteq V$ and $E'=\mset{ e \cap V': e \in E''}$ for some submultiset $E''$ of $E$. For any subset $V'\subseteq V$, we define the {\em subhypergraph of $H$ induced by} $V'$ to be the hypergraph $(V',E')$ with $E' = \mset{e\cap V': e\in E, e\cap V'\ne\emptyset}$. Thus, we obtain the subhypergraph induced by $V'$  by deleting all vertices in $V - V'$ from $V$ and from each edge of $H$, and subsequently deleting all empty edges. By $H \backslash V'$ we denote  the subhypergraph of $H$ induced by $V - V'$, and for $v \in V$, we write shortly $H\backslash v$ instead of $H\backslash \{ v\}$.
Similarly, for any subset $E'\subseteq E$, we denote the subhypergraph $(V, E - E')$ of $H$ by $H - E'$, and for $e \in E$, we write $H-e$ instead of $H- \{ e\}$.

A $k$-length $(v_0,v_k)$-{\em walk} in a hypergraph $H$ is an alternating sequence $W = v_0e_1v_1\ldots$ $ v_{k-1}e_kv_k$ of (possibly repeated) vertices and edges such that $v_0,\ldots,v_k\in V$, $e_1,\ldots,e_k\in E$, and for each $i\in\lbrace 1,\ldots,k\rbrace$,
the vertices $v_{i-1}$ and $v_i$ are adjacent in $H$ via the edge $e_i$.
Note that since adjacent vertices are by definition distinct, no two consecutive vertices in a walk can be the same. It follows that no walk in a hypergraph may contain an edge of cardinality less than 2. The vertices in $V_a(W) = \lbrace v_0,\ldots,v_k\rbrace$ are called the {\em anchors} of $W$, $v_0$ and $v_k$ are the {\em endpoints} of $W$, and $v_1,\ldots,v_{k-1}$ are the {\em internal vertices} of $W$. We also define the edge set of $W$ as $E(W)=\lbrace e_1,\ldots,e_k\rbrace$, and the set of {\em anchor flags} of $W$ as $F(W)=\{ (v_0,e_1), (v_1,e_1), (v_2, e_2), \ldots, (v_{k-1},e_k), (v_k, e_k) \}$.
Walks $W$ and $W'$ in a hypergraph $H$ are said to be {\em edge-disjoint} if $E(W)\cap E(W') = \emptyset$, and {\em anchor-disjoint} if $V_a(W)\cap V_a(W') = \emptyset$.

A walk $W = v_0e_1v_1\ldots v_{k-1}e_kv_k$  is called {\em closed} if $v_0 = v_k$ and $k \ge 2$;  a {\em (strict) trail} if the edges $e_1,\ldots,e_k$ are pairwise distinct; a {\em path} if it is a trail and the vertices $v_0,\ldots,v_k$ are pairwise distinct; and a {\em cycle} if it is a closed  trail and the vertices $v_0,\ldots,v_{k-1}$ are pairwise distinct. Note that a strict trail as defined above has no repeated anchor flags. In \cite{HGconnection}, a walk with this property, but possibly with repeated edges, was called a {\em trail}. In this paper, we shall consider only strict trails, and hence use the shorter term ``trail'' to mean a ``strict trail''.

A walk $W = v_0e_1v_1\ldots v_{k-1}e_kv_k$ is said to {\em traverse} a vertex $v$ and edge $e$ if $v \in V_a(W)$ and $e \in E(W)$, respectively. More precisely, $W$ traverses $e \in E$ exactly $t$ times if $e=e_i$  for exactly $t$ indices $i \in \{ 1,\ldots,k\}$, and traverses $v\in V$ exactly $t$ times if $v=v_i$  for exactly $t$ indices $i \in \{ 1,\ldots,k\}$ in the case $W$ is closed, and exactly $t$ indices $i \in \{ 0,1,\ldots,k\}$ otherwise.

Vertices $u$ and $v$ are {\em connected} in a hypergraph $H$ if there exists a $(u,v)$-walk (equivalently, a $(u,v)$-path \cite[Lemma 3.9]{HGconnection}) in $H$, and
$H$ itself is {\em connected} if every pair of vertices in $V$ are connected in $H$. The {\em connected components} of $H$ are the maximal connected subhypergraphs of $H$ without empty edges. The number of connected components of $H$ is denoted by $c(H)$.

An {\em Euler tour} of a hypergraph $H$ is a closed trail of $H$ traversing every edge of $H$. An {\em Euler family} of $H$ is a set of pairwise edge-disjoint and anchor-disjoint closed trails of $H$ jointly traversing every edge of $H$.

\section{Spanning Euler tours and spanning Euler families}\label{sec:SETs-and-SEFs}

\begin{defn}\label{def:span-Euler-type-struct}{\rm
An Euler tour $T$ of a hypergraph $H$ is said to be {\em spanning} if every vertex of $H$ is an anchor of $T$.
An Euler family $\F$ of a hypergraph $H$ is said to be {\em spanning} if every vertex of $H$ is an anchor of exactly one trail in $\F$.}
\end{defn}

A (spanning) Euler tour may be thought of as a (spanning) Euler family consisting of a single closed trail; however, a hypergraph may admit a spanning Euler family but no spanning Euler tour (see Figure~\ref{fig:EG3}).

Observe that a hypergraph admits a spanning Euler family if and only if each of its connected components admits a spanning Euler family. Therefore, we may limit our investigation of spanning Euler families to connected hypergraphs, and since  empty edges have no impact on connectedness, we shall assume our hypergraphs have no empty edges.

\subsection{General existence results}\label{sec:general}

\begin{figure}[t]
\centerline{\includegraphics[scale=1.00]{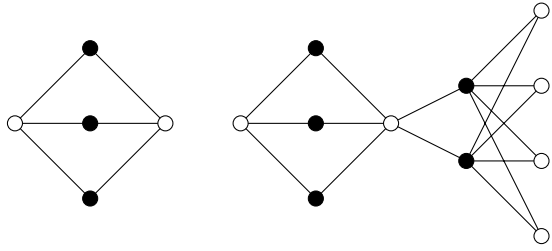}}
\caption{The incidence graphs of two (connected) hypergraphs that admit an Euler family but no spanning Euler family (v-vertices are coloured black).}
\label{fig:EG1}
\end{figure}

Clearly, any hypergraph with a spanning Euler family also admits an Euler family. The converse, however, does not hold, as illustrated by the two examples in  Figure~\ref{fig:EG1}.
A general example is obtained from any hypergraph $H$ that admits an Euler family by adjoining a new vertex to a single edge of $H$; the resulting hypergraph has an Euler family but no spanning Euler family, because no closed trail can traverse a pendant vertex.
These examples suggest some basic necessary conditions for a hypergraph to have a spanning Euler family.

\begin{lemma}\label{lem:evident-nec-cond}
Let $H = (V,E)$ be a hypergraph. If $H$ admits a spanning Euler family, then
\begin{enumerate}[(i)]
\item $|e| \ge 2$ for all $e\in E$,
\item $\deg_H(v) \geq 2$ for all $v\in V$,	
\item $2 \le |V| \le  |E|$, and
\item for each subset $E'\subseteq E$ such that $|E'| = k \ge 2$,  the cardinality of the set
$$\left\{ v \in \bigcap_{e\in E'} e:  \deg_H(v) = k \right\}$$ is at most $k$.
	\end{enumerate}
\end{lemma}

\begin{proof}
(i) and (ii) are easy to see.
\begin{enumerate}[(i)]
\item[(iii)] Let $\mathcal{F}$ be a spanning Euler family of $H$. Since each closed trail in $\mathcal{F}$ traverses at least 2 vertices, we have $|V| \ge 2$. Next, observe that if $T = v_0e_1v_1\cdots v_{n-1}e_nv_0$ is a closed trail in $H$, then $|E(T)| = n$ and $|V_a(T)| \leq n$. Since each vertex and each edge of $H$ occur in exactly one of the closed trails in $\mathcal{F}$, it follows that $| V|\leq| E|$.
\item[(iv)] Let  $E' \subseteq E$ be such that $|E'|=k \ge 2$. Define the set
$S=\left\{ v \in \bigcap_{e \in E'} e: \deg_H(v)=k \right\},$
and let $\ell=|S|$. Take any spanning Euler family $\F$ of $H$ and let $F(\F)$ be the set of all flags traversed by the closed trails in $\F$. Observe that exactly $2k$ flags in $F(\F)$ contain edges of $E'$, and at least $2\ell$ flags in $F(\F)$ contain vertices in $S$. Since every flag containing a vertex in  $S$ must also contain an edge in $E'$, it follows that $2\ell \le 2k$.
\end{enumerate}
\vspace*{-1.0cm}
\end{proof}

Observe that the examples in Figure~\ref{fig:EG1} fail Conditions (iii) and (iv), respectively. It is natural to ask whether  the necessary conditions in Lemma~\ref{lem:evident-nec-cond} are also sufficient. The example in Figure~\ref{fig:EG2} shows that this is not the case. Observe that this hypergraph, call it $H$, has a cut vertex $v$, and that for one of the connected components of $H\backslash v$, call it $H_1$, neither $H_1$ nor the subhypergraph of $H$ induced by $V(H_1) \cup \{v\}$ admits a spanning Euler family (see Theorem~\ref{the:cutvx}). This example suggests that, in order to determine whether a hypergraph admits a spanning Euler family, it is important to consider its cut vertices (and, more generally, vertex cuts), which will be the topic of the remaining sections.

\begin{figure}[b]
\centerline{\includegraphics[scale=1.00]{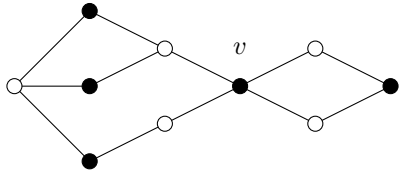}}
\caption{The incidence graph of a hypergraph that satisfies the necessary conditions in Lemma~\ref{lem:evident-nec-cond} but has no spanning Euler family (v-vertices are coloured black). }\label{fig:EG2}
\end{figure}

We briefly mention two other characterisations of hypergraphs admitting spanning Euler families (and tours). The first one, a characterisation in terms of the incidence graph, is analogous to the characterisation of hypergraphs admitting general Euler families and Euler tours; see \cite[Theorem 2.18]{eulerianHG} and \cite[Lemma 1]{Quasi-eulerianHG}. The proof is similar, straightforward, and hence omitted.

\begin{theo}\label{thm:SEF-inc-graph-char}
Let $H = (V,E)$ be a non-empty connected hypergraph, and $G = \mathcal{G}(H)$ its incidence graph. Then, $H$ has a spanning Euler family  if and only if $G$ has a subgraph $G'$ in which every e-vertex is of degree 2 and every v-vertex is of positive even degree. Moreover, $H$ has a spanning Euler tour if and only if $G$ has such a subgraph with a single connected component.
\end{theo}

The second characterisation of hypergraphs admitting spanning Euler families can be obtained just as in  \cite[Corollary 6.2]{Quasi-eulerianHG} using Lovasz's Theorem \cite{Lov}, which gives necessary and sufficient conditions for the existence of a $(g,f)$-factor in a graph.  The resulting  necessary and sufficient conditions for a hypergraph to admit a spanning Euler family are rather complex and not easily verifiable.

Since the problem of existence of an Euler tour is NP-complete even on some restricted families of hypergraphs \cite{lonc-naroski}, so is the problem of existence of a spanning Euler tour. In contrast, the problem of existence of an Euler family is polynomial on the set of all hypergraphs  \cite[Theorem 7.2]{Quasi-eulerianHG}, and the same result for spanning Euler families can be proved in a very similar way, using a polynomial conversion to the problem of existence of an $f$-factor in a graph.

\subsection{Vertex cuts}\label{sec:vx-cuts}

As we shall see, vertex cuts  (to be defined below) play an important role in the existence of spanning eulerian substructures in a hypergraph.

\begin{defn}{\rm
Let $H = (V,E)$ be a hypergraph. A subset $S\subsetneq V$ is said to be a {\em vertex cut} of $H$ if $H\backslash S$ is disconnected. A {\em k-vertex cut} of $H$ is a vertex cut of cardinality $k$.}
\end{defn}

Recall that a {\em cut vertex} in a hypergraph $H$ with at least 2 vertices is a vertex $v$ such that $c(H\backslash v) > c(H)$ \cite[Definition 3.22]{HGconnection}. Thus, in a connected hypergraph, cut vertices correspond precisely to 1-vertex cuts. A vertex cut $S$ of $H$ is said to be {\em minimal} if no proper subset of $S$ is a vertex cut of $H$.

The following basic observations on (minimal) vertex cuts in a hypergraph will be helpful in the proofs of our main results. For an edge $e$ of a hypergraph $H=(V,E)$, we say that $e$ {\em intersects} a set $S \subseteq V$ if $e \cap S \ne \emptyset$, and $e$ {\em intersects} a subhypergraph $H'$ of $H$ if $e \cap V(H') \ne \emptyset$.

\begin{lemma}\label{lem:2.2.2-2.2.4}
Let $H=(V,E)$ be a connected hypergraph and $S$ a  vertex cut of $H$. Then the following hold.
\begin{enumerate}[(1)]
\item Every edge intersecting $S$ contains vertices from at most one connected component of $H\backslash S$.
\item If $S$ is a minimal vertex cut, then
\begin{enumerate}[(a)]
\item each vertex in $S$ is adjacent in $H$ to at least one vertex in each connected component of $H\backslash S$; and
\item $2 \le c(H\backslash S)\le \min \{ \deg_{H'}(v): v \in S \}$, where $H'=H-E'$ and $E'=\{e \in E: e \subseteq S\}$.
\end{enumerate}
\end{enumerate}
\end{lemma}

\begin{proof}
Let $H_1,\ldots,H_k$ be the connected components of $H\backslash S$, so $c(H\backslash S)=k$.
\begin{enumerate}[(1)]
\item Suppose there exists an edge $e$ intersecting $S$ such that $e \cap V(H_i) \ne \emptyset$ and $e \cap V(H_j) \ne \emptyset$ for $i \ne j$. Then $e'=e-S$ is an edge of $H \backslash S$ intersecting two of its connected components, a contradiction.
\item Assume $S$ is a minimal vertex cut.
\begin{enumerate}[(a)]
\item Take any $v \in S$ and let $S'=S -\{v\}$. Suppose there exists $i \in \{1,2,\ldots,k\}$ such that $N_H(v) \cap V(H_i)=\emptyset$. Then, for any $u \in V(H_i)$, we have that $N_H(u) \subseteq V(H_i) \cup S'$. It follows that $S'$ is a vertex cut of $H$, contradicting the minimality of $S$. Hence $N_H(v) \cap V(H_i) \ne \emptyset$ for all $i$.
\item Clearly, $k \ge 2$. It remains to show that $k \le d$, where $d=\min \{ \deg_{H'}(v): v \in S \}$. Let $v \in S$ be such that $\deg_{H'}(v)=d$. By (a), vertex $v$ is adjacent in $H$ to at least one vertex in each of $H_1,\ldots,H_k$, and by (1), each of the edges incident with $v$ in $H$ intersects at most one $H_i$. Since the edges in $E'$ intersect no $H_i$, we have $k \le d$.
\end{enumerate}
\end{enumerate}
\vspace*{-1.0cm}
\end{proof}

In the remainder of this paper,  we shall focus on the impact of particular vertex cuts on the  existence of a spanning Euler family or tour in a hypergraph $H$. First, we define the hypergraphs related to $H$ and its chosen vertex cut $S$ that will play a crucial role in the reduction of the problem.

\begin{defn}{\rm
Let $H = (V,E)$ be a connected hypergraph and $S\subsetneq V$ a vertex cut of $H$. Let $F$ be any connected component of $H\backslash S$.
\begin{itemize}
\item The {\em S-component $F'$ of $H$ corresponding to $F$} is the subhypergraph $(V',E')$ of $H$ with $V'=V(F)\cup S$ and $E' = \mset{ e\in E: e \subseteq V', e \ne \emptyset }$.
\item If $|S|$ is even, we define the {\em S*-component $F^*$ of $H$ corresponding to $F$} as the hypergraph obtained from the $S$-component $F'$ corresponding to $F$ by adjoining $\frac{|S|}{2}$ copies of the edge $S$.
\end{itemize}
}
\end{defn}

For a hypergraph $H$ with a vertex cut $S$ such that $c(H\backslash S) = k$, we generally denote the connected components of $H\backslash S$ by $H_1,\ldots,H_k$, the corresponding $S$-components of $H$ by $H_1',\ldots,H_k'$, and the corresponding $S^*$-components (if $\vert S\vert$ is even) by $H_1^*,\ldots,H_k^*$. We shall refer to all these hypergraphs as the {\em derived hypergraphs of $H$ (with respect to the vertex cut $S$)}.

\subsection{Minimal vertex cuts consisting of vertices of degree two}\label{sec:vx-cuts-deg-2}

We are now ready for our first main result --- the characterisation of hypergraphs with minimal vertex cuts consisting of vertices of the smallest degree possible that admit spanning Euler families.

\begin{theo}\label{the:deg2}
Let $H$ be a connected hypergraph with a minimal vertex cut $S$ such that $\deg_H(v)=2$ for all $v \in S$.
\begin{enumerate}[(1)]
\item If $|S|$ is odd, then $H$ admits no spanning Euler family.
\item If $|S|$ is even, then $H$ admits a spanning Euler family if and only if both $S^\ast$-components of $H$ admit a spanning Euler family,
\end{enumerate}
\end{theo}

\begin{proof}
 From Lemma~\ref{lem:2.2.2-2.2.4}, it follows that $H \backslash S$ has exactly two connected components, $H_1$ and $H_2$, that the sets $E_i=\{ e \in E: e \cap V(H_i) \ne \emptyset \}$, for $i=1,2$, form a partition of $E$, and each vertex in $S$ is incident with exactly one edge from each of $E_1$ and $E_2$.

\begin{enumerate}[(1)]
\item Assume $|S|$ is odd, and suppose $\F$ is a spanning Euler family of $H$. Let $T$ be any closed trail in $\F$ that traverses at least one vertex in $S$. Then $T$ must be of the form $v_0 T_0 v_1 T_1 v_2 \ldots v_{n-1}T_{n-1}v_0$ where $v_0,\ldots,v_{n-1} \in S$ and, for each $i \in \ZZ_n$, $T_i$ is a $(v_i,v_{i+1})$-trail with no internal vertices in $S$. It follows that $E(T_0), E(T_1), \ldots,E(T_{n-1})$ are alternately contained in $E_1$ and $E_2$, whence $n$ must be even. Since the trails in $\F$ are pairwise anchor-disjoint, jointly traverse all vertices in $S$, and each traverses an even number of vertices in $S$, it follows that $|S|$ is even, a contradiction.

\item Assume $|S|$ is even, so that the two $S^\ast$-components $H_1^\ast$ and $H_2^\ast$ are well-defined.

Let $\F$ be a spanning Euler family of $H$. We show that $H_1^\ast$ admits a spanning Euler family (the proof for $H_2^\ast$ is analogous). Let $\F_1$ be the set of closed trails in $\F$ traversing edges in $E_1$. For any $T \in \F_1$, we construct a closed trail $T'$ in $H_1^\ast$ as follows. If $T$ traverses no vertices of $S$, we let $T'=T$. Otherwise, as above, $T$ is of the form $v_0 T_0 v_1 T_1 v_2 \ldots v_{n-1}T_{n-1}v_0$ where $n$ is even, $v_0,\ldots,v_{n-1}$ are the only anchors of $T$ in $S$, and the edge sets of the $(v_i,v_{i+1})$-subtrails $T_i$  are alternately contained in $E_1$ and $E_2$. Obtain $T'$ by replacing each subtrail $T_i$ that traverses edges of $E_2$ with a new copy of the edge $S$ in $H_1^\ast$. Thus $T'$ traverses precisely the same edges in $E_1$ as $T$, together with $\frac{n}{2}$ copies of the edge $S$. Since the trails in $\F_1$ jointly traverse each vertex in $S$ exactly once, the closed trails in $\F_1'=\{ T': T \in \F_1 \}$ jointly traverse each of the $\frac{|S|}{2}$ copies of the edge $S$ in $H_1^\ast$ exactly once, in addition to every edge of $E_1$. It follows that $\F_1'$ is a spanning Euler family of $H_1^\ast$.

Conversely, assume  both $S^\ast$-components admit a spanning Euler family. Fix $i \in \{ 1,2\}$, and let $\F_i$ be a spanning Euler family of $H_i^\ast$. Consider any $T \in \F_i$ that traverses vertices of $S$. Since every vertex $v \in S$ is incident in $H_i^\ast$ with exactly one edge in $E_i$ and $\frac{|S|}{2}$ copies of the edge $S$, it follows that $T$ is of the form $v_0 T_0 v_1 S v_2 \ldots v_{n-2}T_{n-2}v_{n-1}Sv_0$ where $n$ is even, $v_0,\ldots,v_{n-1}$ are the only anchors of $T$ in $S$, and the edge sets of the subtrails $T_0,T_2,\ldots,T_{n-2}$  are all contained in $E_1$. Let $\T_i$ be the family of all subtrails $T_j$ that occur in any closed trail $T \in \F_i$ traversing vertices of $S$, and observe that each vertex $v \in S$ occurs as an endpoint of exactly one subtrail in $\T_i$. It follows that the subtrails in $\T_1 \cup \T_2$ can be suitably concatenated to form a family $\F_S$ of anchor- and edge-disjoint closed trails in $H$. Finally, if we let $\F^\ast$ be the family of all trails in $\F_1 \cup \F_2$ that traverse no vertices of $S$, then $\F_S \cup \F^\ast$ is a spanning Euler family of $H$.
\end{enumerate}
\vspace*{-1.5cm}
\end{proof}

A minor modification to the above proof yields a weaker result for spanning Euler tours.

\begin{cor}\label{cor:deg2}
Let $H$ be a connected hypergraph with a minimal vertex cut $S$ such that $\deg_H(v)=2$ for all $v \in S$. If $H$ admits a spanning Euler tour, then $|S|$ is even and both $S^\ast$-components of $H$ admit a spanning Euler tour.
\end{cor}

\begin{figure}[t]
\centerline{\includegraphics[scale=0.7]{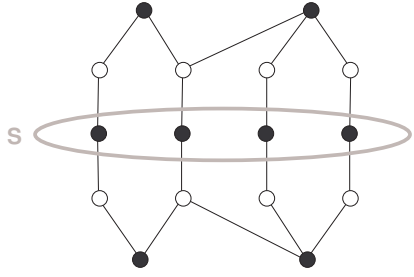}}
\caption{The incidence graph of a hypergraph $H$ with a vertex cut $S$ such that $H$ admits a spanning Euler family and both $S^\ast$-components admit spanning Euler tours, but $H$ does not admit a spanning Euler tour (v-vertices are coloured black). }\label{fig:EG3}
\end{figure}

Note that the converse of Corollary~\ref{cor:deg2} does not hold; an example is shown in Figure~\ref{fig:EG3}.

\subsection{Cut vertices}\label{sec:cut-vx}

We next examine spanning Euler families and tours in connected hypergraphs with vertex cuts of smallest possible cardinality.

\begin{theo}\label{the:cutvx}
Let $H=(V,E)$ be a connected hypergraph with a cut vertex $v$. Let $H_1,\ldots,H_k$ be the connected components of $H \backslash v$, and $H_1',\ldots,H_k'$ the corresponding $\{v \}$-components. Then $H$ has a spanning Euler family  if and only if
\begin{enumerate}[(1)]
\item for some $i \in \{ 1,\ldots,k\}$, the $\{v \}$-component $H_i'$ admits a spanning Euler family , and
\item for each $i \in \{ 1,\ldots,k\}$, at least one of $H_i$ and $H_i'$ admits a spanning Euler family.
\end{enumerate}
\end{theo}

\begin{proof}
Assume $H$ has a spanning Euler family $\F$. We show that (1) and (2) hold.
\begin{enumerate}[(1)]
\item Let $T$ be the unique trail in $\F$ that traverses $v$. Then $T$ is equivalent (see Definition~\ref{def:equiv}) to a concatenation of $(v,v)$-trails $T_1,\ldots,T_k$, where for each index $j$, the trail $T_j$ is contained in the $\{v \}$-component $H_j'$ but may be trivial (of length 0). Since $T$ is non-trivial, there exists an index $i$ such that $T_i$ is non-trivial. Obtain a closed trail $T'$ from $T$ by deleting all $(v,v)$-subtrails not contained in $H_i'$. Let $\F_i$ be the set of all closed trails in $\F -\{ T\}$ that traverse vertices in $H_i'$ (and hence do not traverse $v$). Then $\F_i \cup \{ T'\}$ is a spanning Euler family of $H_i'$.
\item Let $i \in \{ 1,\ldots,k\}$ be such that $H_i'$ does not admit a spanning Euler family, and let $\F_i$ be the set of all closed trails traversing vertices in $H_i$. By the proof of (1), no trail in $\F_i$ traverses $v$, and hence every trail in $\F_i$ traverses only vertices in $H_i$ and edges in $H_i'$. By replacing each  edge $e$ in the trails of $\F_i$ with $e -\{ v \}$, we obtain a spanning Euler family of $H_i$.
\end{enumerate}

Conversely, assume that (1) and (2) hold. Let $\ell \in \{ 1,\ldots,k\}$ be such that $H_{\ell}'$ admits a spanning Euler family $\F_{\ell}$, while for each $i \ne \ell$, let $\F_i$ be a spanning Euler family of either $H_i$ or $H_i'$. For each $i \in \{ 1,\ldots,k\}$, we construct a family $\F_i'$ of closed trails in $H$ as follows.

If $\F_i$ is a spanning Euler family of $H_i'$, then let $\F_i'=\F_i$. Otherwise, obtain $\F_i'$ from $\F_i$ by replacing each edge $e$ in the trails of $\F_i$ with the corresponding edge $e' \in E$, so that either $e'=e$ or $e'= e \cup \{ v\}$. It can then be verified that $\F=\bigcup_{i=1}^k \F_i'$ is a family of edge-disjoint closed trails in $H$ that traverse every vertex and every edge of $H$. By appropriately concatenating the trails in $\F$ we obtain a spanning Euler family of $H$.
\end{proof}

With a slight modification to the above proof we obtain the analogous result for spanning Euler tours.

\begin{cor}\label{cor:cutvx}
Let $H=(V,E)$ be a connected hypergraph with a cut vertex $v$.   Then $H$ has a spanning Euler tour  if and only if every $\{v \}$-component of $H$ admits a spanning Euler tour.
\end{cor}

By virtue of this result, if a hypergraph $H = (V,E)$ has a cut vertex $v$, the problem of determining whether or not $H$ admits a spanning Euler family (tour) can be reduced to the equivalent problem on the $\lbrace v\rbrace$-components of $H$ and the connected components of $H\backslash v$. This reduction can be applied recursively, and so we need only solve the problem of determining whether particular subhypergraphs of $H$ without cut vertices have a spanning Euler family (tour).

\subsection{Vertex cuts of cardinality two}\label{sec:2-vx-cuts}

We shall now consider the effect of 2-vertex cuts on the existence of spanning Euler families and tours. First, we need to develop some additional concepts.
Recall that for any trail $T$, the anchor flags in $F(T)$ are pairwise distinct. Hence the following definition makes sense.

\begin{defn}{\rm
Let $T$ be a trail in a hypergraph $H$. The {\em incidence graph} of $T$, denoted $\G(T)$,  is the subgraph of $\G(H)$ with vertex set $V_a(T) \cup E(T)$ and edge set $F(T)$. Furthermore, if $\T$ is a family of pairwise edge-disjoint trails in $H$, then we define the {\em incidence graph} of $\T$, denoted $\G(\T)$,  as the (edge-disjoint) union of the incidence graphs of all the trails in $\T$.
}
\end{defn}

\begin{defn}\label{def:equiv}{\rm
Let $\T$ and $\T'$ be two families of pairwise edge-disjoint trails in a hypergraph $H$. We call $\T$ and $\T'$ {\em equivalent}, denoted $\T \equiv \T'$, if $\G(\T)=\G(\T')$.
}
\end{defn}

Clearly, $\equiv$ is an equivalence relation on the set of families of pairwise edge-disjoint trails of $H$. We now take a look at families of cycles, which will be particularly useful in the rest of this paper.

\begin{defn}{\rm
Let $\T$  be a family of pairwise edge-disjoint closed trails in a hypergraph $H$, and $\C$ a family of pairwise edge-disjoint cycles in $H$. Then $\C$ is said to be a {\em cycle decomposition} of $\T$ if  $\T \equiv \C$.

A family of pairwise edge-disjoint cycles that jointly traverse every vertex and every edge of $H$ is called a {\em spanning cycle decomposition} of  $H$.
}
\end{defn}

With the help of \cite[Lemma 3.6]{HGconnection} and Theorem~\ref{thm:SEF-inc-graph-char}, the following observations are then easy to see.

\begin{lemma}\label{lem:CD}
Let $H$ be a hypergraph.
\begin{enumerate}[(1)]
\item Every family of pairwise edge-disjoint closed trails in $H$ admits a cycle decomposition.
\item A family of cycles of $H$ is a spanning cycle decomposition of $H$ if and only if it is a cycle decomposition of a spanning Euler family of $H$.
\item A spanning cycle decomposition $\C$ of $H$ is a cycle decomposition of an Euler tour of $H$ if and only if $\G(\C)$ is connected.
\end{enumerate}
\end{lemma}

\begin{proof}
\begin{enumerate}[(1)]
\item Let $\T$  be any family of pairwise edge-disjoint closed trails in $H$. The incidence graph of $\T$ is an even subgraph of $\G(H)$ and hence admits a cycle decomposition $\C_G$. The family of cycles in $H$ that correspond to the cycles in $\C_G$ forms a cycle decomposition of $\T$.
\item Let $\C$ be a spanning cycle decomposition of $H$. The incidence graph of $\C$ is a spanning even subgraph of $\G(H)$, hence its  connected components $G_1,\ldots,G_k$ admit Euler tours $T_1,\ldots,T_k$, respectively. Let $T_1^H,\ldots,T_k^H$ be the closed trails of $H$ corresponding to $T_1,\ldots,T_k$, respectively. Then $\F=\{ T_1^H,\ldots,T_k^H \}$ is a spanning Euler family of $H$, and $\F \equiv \C$. Hence $\C$ is a cycle decomposition of a spanning Euler family of $H$. The converse follows straight from the definition.

\item If $\G(\C)$ is connected, then in the proof of (2) we have $k=1$, so $T_1$ is a spanning Euler tour of $H$. Conversely, if $\C$ is a cycle decomposition of a spanning Euler tour $T$ of $H$, then $\G(\C)=\G(T)$, so $\G(\C)$ is connected.
\end{enumerate}
\vspace*{-1.5cm}
\end{proof}

We shall now develop a tool that will allow us to classify closed trails in a hypergraph relative to a vertex subset of cardinality 2.

\begin{defn}{\rm
Let $H=(V,E)$ be a hypergraph, and $S \subseteq V$ such that $|S|=2$. Furthermore, let $E_S=\mset{ e \in E: e=S }$.

For a closed trail $T$ in $H$, we define parameters $a(T)$, $b(T)$, and $c(T)$ as follows:
\begin{itemize}
\item $a(T)$ is the number of times vertices in $S$ are traversed by $T$ (that is, the number of vertices in the sequence $T$ that lie in $S$, counting the endpoints of $T$ as one occurrence);
\item $b(T)$ is the number of edges of $E_S$ traversed by $T$; and
\item $c(T)$ is the number of connected components $H_i$ of $H \backslash S$ such that $T$ traverses at least one edge intersecting $H_i$.
\end{itemize}
The triple $(a(T),b(T),c(T))$ is called the {\em $S$-type} of the trail $T$.
}
\end{defn}

Observe that if a trail traverses a vertex in a connected component $H_i$ of $H \backslash S$, then it must also traverse an edge intersecting $H_i$, and recall from Lemma~\ref{lem:2.2.2-2.2.4} that no edge of $H$ intersects more than one connected component of $H \backslash S$.

%%%%

\begin{figure}
\centerline{\includegraphics[scale=1.00]{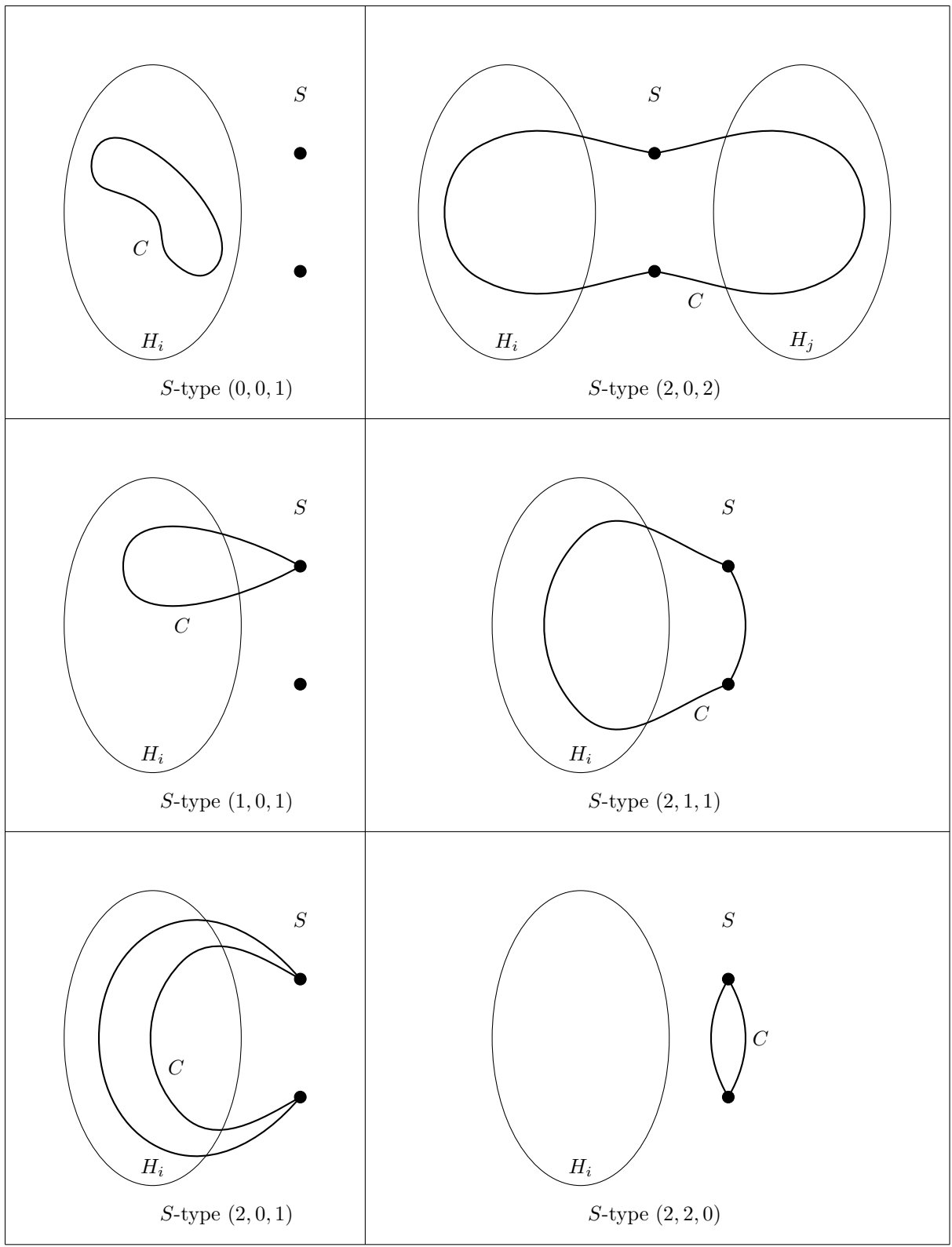}}
\caption{Cycles of different $S$-types. }\label{fig:S-types}
\end{figure}

%%%%

The following is easy to establish, and cycles of different $S$-types are illustrated in Figure~\ref{fig:S-types}.

\begin{lemma}
Let $H=(V,E)$ be a hypergraph, and $S \subseteq V$ such that $|S|=2$. Furthermore, let $E_S=\mset{ e \in E: e=S }$, and let $C$ be a cycle of $H$. Then the $S$-type of $C$ is in the set $\{ (0,0,1),(1,0,1),(2,0,1),(2,0,2),(2,1,1),(2,2,0)\}$.
\end{lemma}

We say that a cycle {\em requires completion (with respect to $S$)} if it is of $S$-type $(2,0,2)$ or $(2,1,1)$. The following lemma will be our most important tool in the proofs of Theorems~\ref{the:even}--\ref{the:oddSET}.

\begin{lemma}\label{lem:RC}
Let $H=(V,E)$ be a hypergraph, $S \subseteq V$ such that $|S|=2$,  $E_S=\mset{ e \in E: e=S }$, and $H_1,\ldots,H_k$ the connected components of $H \backslash S$. Then any spanning Euler family of $H$ admits a cycle decomposition $\C$ such that:
\begin{enumerate}[(1)]
\item for each $i \in \{1,\ldots,k\}$, at most one cycle in $\{ C \in \C: e \cap V(H_i) \ne \emptyset \mbox{ for some } e \in E(C) \}$ requires completion,
\item if $|E_S|$ is even, no cycle in $\C$ is of type $(2,1,1)$, and
\item if $|E_S|$ is odd, exactly one cycle in $\C$  is of type $(2,1,1)$.
\end{enumerate}
\end{lemma}

\begin{proof}
Let $\F$ be any spanning Euler family of $H$. For any cycle decomposition $\C$ of $\F$ (which exists by Lemma~\ref{lem:CD}) and every $i \in \{1,\ldots,k\}$, we denote $\C_i=\{ C \in \C: e \cap V(H_i) \ne \emptyset \mbox{ for some } e \in E(C) \}$. Furthermore, we define $RC(\C,i)$ as the number of cycles in $\C_i$ that require completion.

Let $\C$ be a cycle decomposition of $\F$ that minimises $\sum_{i=1}^k RC(\C,i)$, and suppose that $RC(\C,j) \ge 2$  for some $j \in \{1,\ldots,k\}$. Let $C,C' \in \C_j$ be two cycles requiring completion.

If $C$ and $C'$ are both of $S$-type (2,1,1), then $\{C,C'\}\equiv \{T,C''\}$ where $T$ is a closed trail of $S$-type (2,0,1) and $C''$ is a cycle of $S$-type (2,2,0). Furthermore, $\{ T \}$ has a cycle decomposition $\C_T$ containing no cycles requiring completion. Replacing the subset $\{C,C'\}$ of $\C$ with $\C_T \cup \{C''\}$, we thus obtain a cycle decomposition $\C'$ of $\F$ such that $RC(\C',j)=RC(\C,j)-2$ and $RC(\C',i)=RC(\C,i)$ for all $i \ne j$, contradicting the choice of $\C$.

If $C$ and $C'$ are both of $S$-type (2,0,2), then $\{C,C'\}\equiv \{T,T'\}$ where $T$ is a closed trail of $S$-type (2,0,1) and $T'$ is a closed trail of $S$-type (2,0,1) or (2,0,2). The first case occurs when $C,C' \in \C_j \cap \C_{\ell}$ for some $\ell \ne j$. Replacing the subset $\{C,C'\}$ of $\C$ with the union of cycle decompositions of $\{T\}$ and $\{T'\}$, we obtain a cycle decomposition $\C'$ of $\F$ such that $RC(\C',i)=RC(\C,i)-2$ for $i \in \{j,\ell\}$ and $RC(\C',i)=RC(\C,i)$ for all $i \not\in \{ j,\ell\}$. The second case occurs when $C \in \C_j \cap \C_{\ell}$ and
$C' \in \C_j \cap \C_{m}$ for $j,\ell,m$ pairwise distinct. We can then construct a cycle decomposition $\C'$ of $\F$ such that $RC(\C',j)=RC(\C,j)-2$ and $RC(\C',i)=RC(\C,i)$ for all $i \ne j$. In both cases, we have a contradiction.

Finally, suppose $C$ is of $S$-type (2,1,1) and $C'$ is of $S$-type (2,0,2). Then $\{C,C'\}\equiv \{T,T'\}$ where $T$ is a closed trail of $S$-type (2,0,1) and $T'$ is a closed trail of $S$-type (2,1,1). If $\ell \ne j$ is such that $C' \in \C_j \cap \C_{\ell}$, then an appropriate replacement results in a cycle decomposition $\C'$ of $\F$ such that $RC(\C',j)=RC(\C,j)-2$ and $RC(\C',i)=RC(\C,i)$ for all $i \ne j$, including $i=\ell$. Again, we have a contradiction with the choice of $\C$.

We conclude that $RC(\C,i) \le 1$ for all $i \in \{1,\ldots,k\}$, so $\C$ satisfies (1).

Now let $\C$ be a cycle decomposition for $\F$ that satisfies (1) and, among such cycle decompositions, also contains the smallest number of cycles of $S$-type (2,1,1).  Observe that the number of cycles in $\C$ of $S$-type (2,1,1) is even if $|E_S|$ is even, and odd otherwise. Hence it remains to show that $\C$ has at most one cycle of $S$-type (2,1,1).

Suppose, to the contrary, that $C_1$, $C_2$ are two cycles in $\C$ of $S$-type (2,1,1), where necessarily $C_1 \in \C_i$ and $C_2 \in \C_j$ for $i \ne j$. Then $\{ C_1,C_2 \} \equiv \{ C_1',C_2'\}$, where $C_1'$ is of $S$-type (2,0,2) and $C_2'$ is of $S$-type (2,2,0).
Replacing the subset $\{C_1,C_2\}$ of $\C$ with $\{C_1',C_2'\}$, we  obtain a cycle decomposition $\C'$ of $\F$ such that $RC(\C',i)=RC(\C,i)$  for all $i$ while $\C'$ contains fewer cycles of $S$-type (2,1,1) --- a contradiction.

Hence $\C$ satisfies Properties (2) and (3) as well.
\end{proof}

%%%%% 2-vertex cuts

We are now ready for our last main result --- the characterisation of hypergraphs with  2-vertex cuts that admit a spanning Euler family (split between Theorems~\ref{the:even} and \ref{the:odd}) or spanning Euler tour (split between Theorems~\ref{the:evenSET} and \ref{the:oddSET}).

%%%%% SEF for |E_S| even

\begin{theo}\label{the:even}
Let $H=(V,E)$ be a hypergraph with a  2-vertex cut $S=\{ u,v \}$,  let $E_S=\mset{ e \in E: e=S }$, and assume $|E_S|$ is even. Let $H_1,\ldots,H_k$ be the connected components of $H \backslash S$; $H_1',\ldots,H_k'$ the corresponding $S$-components; and $H_1^\ast,\ldots,H_k^\ast$ the corresponding $S^\ast$-components. Then $H$ admits a spanning Euler family if and only if there exists an even-size  subset $I \subseteq \{ 1,\ldots,k \}$ such that
\begin{enumerate}[(1)]
\item $H_i^\ast$ admits a spanning Euler family for each $i \in I$;
\item for each $i \not\in I$, at least one of $H_i$, $H_i'$,  $H_i' \backslash u$, and $H_i' \backslash v$ admits a spanning Euler family; and
\item if $I =\emptyset$, then at least one of the following holds:
\begin{enumerate}
\item $H_i'$ admits a spanning Euler family for some $i$;
\item $H_i' \backslash u$ and $H_j' \backslash v$ admit spanning Euler families for some $i \ne j$.
\end{enumerate}
\end{enumerate}
\end{theo}

\begin{proof}
$(\Rightarrow)$ Assume $H$ admits a spanning Euler family $\F$. By Lemma~\ref{lem:RC}, there exists a cycle decomposition $\C$ of $\F$ such that for each $i \in \{ 1,\ldots,k \}$, the set $\C_i=\{ C \in \C: e \cap V(H_i) \ne \emptyset \mbox{ for some } e \in E(C) \}$ contains at most one cycle requiring completion, which is of $S$-type (2,0,2). Let $\C_S=\left\{ C \in \C: E(C)=\{S \} \right\}$ and observe that the cycles in $\C_S$  traverse all edges of $E_S$ (possibly vacuously).
Let $I$ be the set of indices $i \in\{ 1,\ldots,k \}$ such that $\C_i$ has a cycle  of $S$-type (2,0,2), and observe that $|I|$ is even.

Take any $i \in \{ 1,\ldots,k \}$. In each case we construct a spanning cycle decomposition $\C_i'$  of one the hypergraphs $H_i$, $H_i'$, $H_i^\ast$, $H_i' \backslash u$, and $H_i' \backslash v$.

\smallskip

{\em Case (i):} $i \in I$. Then $\C_i$ contains a unique cycle $C_i$ requiring completion, namely, of $S$-type (2,0,2). Let $P_i$ be the unique $(u,v)$-path in $C_i$ whose internal vertices lie in $H_i$. Complete $P_i$ to a cycle $C_i'$ of $S$-type (2,1,1) in $H_i^\ast$ using the additional copy of the edge $S$. Then $\C_i'=(\C_i - \{ C_i \} ) \cup \C_S \cup \{ C_i' \}$ is a spanning cycle decomposition for $H_i^\ast$.

{\em Case (ii):} $i \not\in I$, and $E_S\ne \emptyset$ or $\C_i$ has cycles traversing $u$ and cycles traversing $v$. Then $\C_i'=\C_i \cup \C_S$ is a spanning cycle decomposition for $H_i'$.

{\em Case (iii):} $i \not\in I$, $E_S = \emptyset$, and $\C_i$ has a cycle traversing $u$ but none traversing $v$. Observe that the cycles of $\C_i$ traverse every edge of $H_i'$ and every vertex in $H_i'$ except $v$. Obtain a spanning cycle decomposition $\C_i'$ for $H_i' \setminus v$ by replacing every edge $e$ in every cycle in $\C_i$ with $e - \{ v\}$.

{\em Case (iv):} $i \not\in I$, $E_S = \emptyset$, and $\C_i$ has a cycle traversing $v$ but none traversing $u$. This is analogous to Case (iii).

{\em Case (v):}  $i \not\in I$, $E_S = \emptyset$, and $\C_i$ has no cycles traversing vertices of $S$. Obtain a spanning cycle decomposition $\C_i'$ for $H_i$ by replacing every edge $e$ in every cycle in $\C_i$ with $e - S$.

\smallskip

In each case, we have a spanning cycle decomposition $\C_i'$ of the corresponding hypergraph, so by Lemma~\ref{lem:CD}(2), Statements (1) and (2) follow. Moreover, since $\C$ contains a cycle traversing $u$, there exists an index $i$ such that at least one of Cases (i)-(iii) holds for $i$. And since $\C$ contains a cycle traversing $v$, if Cases (i)-(ii) hold for no index $\ell$, then there exists $j \ne i$ such that Case (iv) holds for index $j$. Hence (3) follows as well.

\medskip

$(\Leftarrow)$ Assume there exists an even-size subset $I \subseteq \{ 1,\ldots,k \}$ such that (1)-(3) hold. If $I \ne\emptyset$, then for each $i \in I$, let $\F_i$ be a spanning Euler family of $H_i^\ast$, and for each $i \not\in I$, let $\F_i$ be a spanning Euler family of one of the hypergraphs $H_i$, $H_i'$,  $H_i' \backslash u$, and $H_i' \backslash v$.

If $I=\emptyset$, then either choose $\ell \in \{1,\ldots,k\}$ such that $H_{\ell}'$ has a spanning Euler family $\F_{\ell}$, or else choose distinct $s,t \in \{1,\ldots,k\}$ such that $H_{s}'\backslash u$ and $H_{t}'\backslash v$ have spanning Euler families $\F_s$ and $\F_t$, respectively. For each $i \ne \ell$ (in the first case) or $i \not\in \{s,t\}$ (in the second case), let $\F_i$ be a spanning Euler family of one of the hypergraphs $H_i$, $H_i'$,  $H_i' \backslash u$, and $H_i' \backslash v$.

Take any $i \in \{ 1,\ldots,k \}$ and let $\C_i$ be a cycle decomposition of $\F_i$. Note that if $\F_i$ is a spanning Euler family of $H_i^\ast$ or $H_i'$, then by Lemma~\ref{lem:RC} we may assume that $\C_i$ contains at most one cycle requiring completion with respect to $S$; namely, a cycle of $S$-type (2,1,1). Let $\C_i^S$ denote the set of cycles of $S$-type (2,2,0) in $\C_i$, and construct a family of cycles $\C_i'$ from $\C_i$ as follows.

\smallskip

{\em Case (i):} $\F_i$ is a spanning Euler family of $H_i^\ast$, that is, $i \in I$. Since $H_i^\ast$ contains an odd number of copies of the edge $S$, we have that $\C_i$ has a unique cycle $C_i$ of $S$-type (2,1,1); this cycle traverses one copy of the edge $S$, while all others are traversed by cycles in $\C_i^S$. Let $\C_i'=\C_i-(\C_i^S \cup \{ C_i \})$. In addition, let $P_i$ be the unique $(u,v)$-path contained in $C_i$.

{\em Case (ii):} $\F_i$ is a spanning Euler family of $H_i'$. Now $H_i'$ contains an even number of copies of the edge $S$, so all of them are traversed by cycles in $\C_i^S$. Let  $\C_i'=\C_i-\C_i^S $.

{\em Case (iii):} $\F_i$ is a spanning Euler family of $H_i' \backslash x$, for $x \in \{u,v\}$. For each edge $e$ of $H_i' \backslash x$, let $e'$ be the corresponding edge of $H_i'$, so that either $e'=e$ or $e'= e \cup \{ x \}$. Obtain $\C_i'$ from $\C_i$ by replacing each edge $e$ in each cycle in $\C_i$ by $e'$.

{\em Case (iv):} $\F_i$ is a spanning Euler family of $H_i$. Obtain $\C_i'$ from $\C_i$ by replacing each edge $e$ in the cycles of $\C_i$ by the corresponding edge $e'\in E$ (so that $e=e'-S$).

\smallskip

Since $|I|$ is even, we can concatenate pairs of paths $P_i$, for $i \in I$, to obtain a family $\C_I$ of $\frac{|I|}{2}$ cycles in $H$. Furthermore, let $\C_S$ be a family of $\frac{|E_S|}{2}$ pairwise edge-disjoint cycles of $S$-type (2,2,0) in $H$. It can then be verified that $\C=\C_I \cup \C_S \cup \bigcup_{i=1}^k \C_i'$ is a spanning cycle decomposition of $H$, so the result follows by Lemma~\ref{lem:CD}(2).
\end{proof}

%%%%% SET for |E_S| even

%%%%% Spannning Euler tours - even |E_S|

The analogous result for spanning Euler tours in Theorem~\ref{the:evenSET} below is proved similarly, hence we shall only highlight the differences. For a vertex cut $S$ in a hypergraph $H$, and a connected component $H_i$ of $H \backslash S$, we additionally define the {\em $S^{\ast\ast}$-component $H_i^{\ast\ast}$ of $H$ corresponding to $H_i$} as the hypergraph obtained from the $S$-component $H_i'$ by adjoining two copies of the edge $S$.

\begin{theo}\label{the:evenSET}
Let $H=(V,E)$ be a hypergraph with a  2-vertex cut $S=\{ u,v \}$,  let $E_S=\mset{ e \in E: e=S }$, and assume $|E_S|$ is even. Let $H_1,\ldots,H_k$ be the connected components of $H \backslash S$; $H_1',\ldots,H_k'$ the corresponding $S$-components; $H_1^\ast,\ldots,H_k^\ast$ the corresponding $S^\ast$-components, and $H_1^{\ast\ast},\ldots,H_k^{\ast\ast}$ the corresponding $S^{\ast\ast}$-components. Then $H$ admits a spanning Euler tour if and only if there exists an even-size subset $I \subseteq \{ 1,\ldots,k \}$ such that
\begin{enumerate}[(1)]
\item $H_i^\ast$ admits a spanning Euler tour for each $i \in I$;
\item for each $i \not\in I$, at least one of  $H_i'$ and  $H_i^{\ast\ast}$ admits a spanning Euler tour; and
\item if $I =\emptyset$, then for some $i$, the hypergraph $H_i'$ admits a spanning Euler tour.
\end{enumerate}
\end{theo}

\begin{proof}
$(\Rightarrow)$ Assume $H$ admits a spanning Euler tour $T$, and let $\F=\{ T \}$. Define $\C$, $\C_i$ (for $i \in\{ 1,\ldots,k \}$), $\C_S$, and $I$ as in the proof of Theorem~\ref{the:even}.
For each $i \in \{ 1,\ldots,k \}$, construct a spanning cycle decomposition $\C_i'$  of one the hypergraphs $H_i'$, $H_i^\ast$, and  $H_i^{\ast\ast}$ as follows.

\smallskip

{\em Case (i):} $i \in I$. This is identical to Case (i) of the proof of Theorem~\ref{the:even}.

{\em Case (ii):} $i \not\in I$ and $E_S\ne \emptyset$. This is identical to Case (ii) of the proof of Theorem~\ref{the:even}.

{\em Case (iii):} $i \not\in I$ and $E_S = \emptyset$. Observe that, since the incidence graph $\G(\C)$ of $\C$ is connected, the cycles in $\C_i$ jointly traverse at least one vertex in $S$. Let $C_S$ be a cycle in $H_i^{\ast\ast}$ of $S$-type (2,2,0). Then $\C_i'=\C_i \cup \{ C_S\}$ is a spanning cycle decomposition for $H_i^{\ast\ast}$.

\smallskip

Since in each case  $\G(\C_i')$ contains a $(u,v)$-path, it is connected. Hence by Lemma~\ref{lem:CD}(3),  $\C_i'$ is a cycle decomposition of a spanning Euler tour of the corresponding  hypergraph, and (1)-(2) follow. Suppose $I =\emptyset$. If $E_S\ne \emptyset$, then clearly $\C_i'$ is a spanning cycle decomposition for $H_i'$ for each $i$. If $E_S = \emptyset$, then $T$ must contain, for some $i$, a $(u,v)$-subtrail that traverses only edges in $H_i'$. Consequently, $\G(\C_i)$ is connected, and hence $\C_i$ itself is a cycle decomposition of a spanning Euler tour of $H_i'$. Thus (3) follows as well.

\medskip

$(\Leftarrow)$ Assume there exists an even-size subset $I \subseteq \{ 1,\ldots,k \}$ such that (1)-(3) hold. If $I\ne \emptyset$, then for each $i\in I$, let $T_i$ be a spanning Euler tour of $H_i^\ast$, and for each $i\not\in I$, let $T_i$ be a spanning Euler tour of one of the hypergraphs $H_i'$ and $H_i^{\ast\ast}$. If $I = \emptyset$, let $\ell$ be such that $H_{\ell}'$ admits a spanning Euler tour $T_{\ell}$ and for each $i \ne \ell$, let $T_i$ be a spanning Euler tour of one of the hypergraphs $H_i'$ and $H_i^{\ast\ast}$.

Take any $i \in \{ 1,\ldots,k \}$, let $\C_i$ be a cycle decomposition of $\{ T_i \}$ containing at most one cycle requiring completion, and proceed as in the proof of Theorem~\ref{the:even}.

\smallskip

{\em Case (i):} $T_i$ is a spanning Euler tour of $H_i^\ast$. This is identical to Case (i) of the proof of Theorem~\ref{the:even}.

{\em Case (ii):} $T_i$ is a spanning Euler tour of $H_i'$ or $H_i^{\ast\ast}$. Now the hypergraph contains an even number of copies of the edge $S$, so all of them are traversed by cycles in $\C_i^S$. Let  $\C_i'=\C_i-\C_i^S $.

\smallskip

Construct a spanning cycle decomposition $\C=\C_I \cup \C_S \cup \bigcup_{i=1}^k \C_i'$ of $H$ as before. Note that each $\G(\C_i')$ has at most two connected components (one containing $u$ and one containing $v$). If $I\ne\emptyset$, then $\C_I$ is non empty and it follows that $\G(\C)$ is connected. If $I = \emptyset$, then $\G(\C_{\ell})$ is connected, contains both of $u$ and $v$, and is a subgraph of $\G(\C)$, whence it follows that  $\G(\C)$ is connected. Hence, by Lemma~\ref{lem:CD}(3), $\C$ is a cycle decomposition of a spanning Euler tour of $H$.
\end{proof}

%%%%% SEF for |E_S| odd

It remains to state and prove results analogous to Theorems~\ref{the:even} and \ref{the:evenSET} for $|E_S|$ odd.

\begin{theo}\label{the:odd}
Let $H=(V,E)$ be a hypergraph with a  2-vertex cut $S=\{ u,v \}$,  let $E_S=\mset{ e \in E: e=S }$, and assume $|E_S|$ is odd. Let $H_1,\ldots,H_k$ be the connected components of $H \backslash S$; $H_1',\ldots,H_k'$ the corresponding $S$-components; and $H_1^\ast,\ldots,H_k^\ast$ the corresponding $S^\ast$-components. Then $H$ admits a spanning Euler family if and only if there exists an odd-size  subset $J \subseteq \{ 1,\ldots,k \}$ such that
\begin{enumerate}[(1)]
\item $H_i'$ admits a spanning Euler family for each $i \in J$; and
\item for each $i \not\in J$, at least one of $H_i$ and $H_i^\ast$ admits a spanning Euler family.
\end{enumerate}
\end{theo}

\begin{proof}
$(\Rightarrow)$ Assume $H$ admits a spanning Euler family $\F$. By Lemma~\ref{lem:RC}, there exists a cycle decomposition $\C$ of $\F$ such that for each $i \in \{ 1,\ldots,k \}$, the set
$\C_i=\{ C \in \C: e \cap V(H_i) \ne \emptyset \mbox{ for some } e \in E(C) \}$ contains at most one cycle requiring completion,  and $\C$ contains exactly one cycle of $S$-type (2,1,1). Let $\C_S=\left\{C \in \C: E(C)=\{ S \} \right\}$ and observe that the cycles in $\C_S$ traverse all but one of the edges of $E_S$.

Let $I$ be the set of indices $i \in\{ 1,\ldots,k \}$ such that $\C_i$ has a cycle  of $S$-type (2,0,2), and observe that $|I|$ is even. Furthermore, let $\ell$ be the unique index such that $\C_{\ell}$ has a cycle  of $S$-type (2,1,1), and let $J=I \cup \{ \ell \}$.

Take any $i \in \{ 1,\ldots,k \}$. In each case we construct a spanning cycle decomposition $\C_i'$ of  one of the hypergraphs $H_i$, $H_i'$, and $H_i^\ast$.

\smallskip

{\em Case (i):} $i \in I$. Then $\C_i$ contains a unique cycle $C_i$ requiring completion, namely, of $S$-type (2,0,2). Let $P_i$ be the unique $(u,v)$-path in $C_i$ whose internal vertices lie in $H_i$. Complete $P_i$ to a cycle $C_i'$ of $S$-type (2,1,1) in $H_i'$ using the copy of the edge $S$ not traversed by any cycle in $\C_S$. Then $\C_i'=(\C_i - \{ C_i \} ) \cup \C_S \cup \{ C_i' \}$ is a spanning cycle decomposition for $H_i'$.

{\em Case (ii):} $i=\ell$. Then $\C_i$ contains a unique cycle $C_i$ requiring completion, namely, of $S$-type (2,1,1). It follows $\C_i'=\C_i  \cup \C_S $ is a spanning cycle decomposition for $H_i'$.

{\em Case (iii):} $i \not\in J$ and $\C_i$ has a cycle traversing some vertex in $S$. Note that none of the cycles in $\C_i$ traverse the edges of $E_S$. Let $\C_S'$ be a family of $\frac{|E_S|+1}{2}$ pairwise edge-disjoint cycles of $S$-type (2,2,0) in $H_i^\ast$. Then $\C_i'=\C_i \cup \C_S'$ is a spanning cycle decomposition for $H_i^\ast$.

{\em Case (iv):}  $i \not\in J$ and $\C_i$ has no cycles traversing vertices of $S$. Obtain a spanning cycle decomposition $\C_i'$ for $H_i$ by replacing every edge $e$ in every cycle in $\C_i$ with $e - S$.

\smallskip

By Lemma~\ref{lem:CD}(2), each $\C_i'$ is a cycle decomposition of a spanning Euler family for the corresponding hypergraph. Hence (1) and (2) follow.

\medskip

$(\Leftarrow)$ Assume there exists an odd-size subset $J \subseteq \{ 1,\ldots,k \}$ such that (1) and (2) hold. For each $i \in J$, let $\F_i$ be a spanning Euler family of $H_i'$, and for each $i \not\in J$, let $\F_i$ be a spanning Euler family of one of the hypergraphs $H_i$ and $H_i^\ast$.

Take any $i \in \{ 1,\ldots,k \}$ and let $\C_i$ be a cycle decomposition of $\F_i$. Note that, since $|E_S|$ is odd, if $\F_i$ is a spanning Euler family of $H_i^\ast$, then we may assume by Lemma~\ref{lem:RC}  that $\C_i$ contains no cycles requiring completion with respect to $S$, and if $\F_i$ is a spanning Euler family of $H_i'$, then $\C_i$ contains exactly one cycle requiring completion; namely, a cycle of $S$-type (2,1,1). Let $\C_i^S$ denote the set of cycles of $S$-type (2,2,0) in $\C_i$, and construct a family of cycles $\C_i'$ from $\C_i$ as follows.

{\em Case (i):} $\F_i$ is a spanning Euler family of $H_i^\ast$, so $i \not\in J$. Note that all $|E_S|+1$ copies of the edge $S$ are traversed by the cycles of $S$-type (2,2,0) in $\C_i^S$.  Let $\C_i'=\C_i-\C_i^S$.

{\em Case (ii):} $\F_i$ is a spanning Euler family of $H_i'$, so $i \in J$. Let $C_i$ be the unique cycle of $S$-type (2,1,1) in $\C_i$, and $P_i$  the unique $(u,v)$-path contained in $C_i$. Let $\C_i'=\C_i-(\C_i^S \cup \{ C_i \})$.

{\em Case (iii):} $\F_i$ is a spanning Euler family of $H_i$. Obtain $\C_i'$ from $\C_i$ by replacing each edge $e$ in the cycles of $\C_i$ by the corresponding edge $e'\in E$ (so that $e=e'-S$).

Choose an index $\ell \in J$ and let $I=J-\{ \ell \}$.
Since $|I|$ is even, we can concatenate pairs of paths $P_i$, for $i \in I$, to obtain a family $\C_I$ of $\frac{|I|}{2}$ cycles in $H$. Complete $P_{\ell}$ to a cycle $C_{\ell}$ by adjoining one copy of the edge $S$, and let  $\C_S$ be a family of $\frac{|E_S|-1}{2}$ pairwise edge-disjoint cycles of $S$-type (2,2,0) in $H$ that jointly traverse the remaining copies of the edge $S$. It can then be verified that $\C=\C_I \cup \{ C_{\ell} \} \cup \C_S \cup \bigcup_{i=1}^k \C_i'$ is a spanning cycle decomposition of $H$.
\end{proof}

Again, the analogous result for spanning Euler tours, Theorem~\ref{the:oddSET} below, requires only a few modifications.

\begin{theo}\label{the:oddSET}
Let $H=(V,E)$ be a hypergraph with a  2-vertex cut $S=\{ u,v \}$,  let $E_S=\mset{ e \in E: e=S }$, and assume $|E_S|$ is odd. Let $H_1,\ldots,H_k$ be the connected components of $H \backslash S$; $H_1',\ldots,H_k'$ the corresponding $S$-components; and $H_1^\ast,\ldots,H_k^\ast$ the corresponding $S^\ast$-components. Then $H$ admits a spanning Euler tour if and only if there exists an odd-size  subset $J \subseteq \{ 1,\ldots,k \}$ such that
\begin{enumerate}[(1)]
\item $H_i'$ admits a spanning Euler tour for each $i \in J$; and
\item $H_i^\ast$ admits a spanning Euler tour for each $i \not\in J$.
\end{enumerate}
\end{theo}

\begin{proof}
$(\Rightarrow)$ Assume $H$ admits a spanning Euler tour $T$, and let $\F=\{ T \}$. Define  $\C$, $\C_i$  (for  $i \in \{ 1,\ldots,k \}$), $\C_S$, $I$, $\ell$, and $J$ as in the proof of Theorem~\ref{the:odd}.

Take any $i \in \{ 1,\ldots,k \}$ and construct a spanning cycle decomposition $\C_i'$ of one of the hypergraphs  $H_i'$ and $H_i^\ast$.

\smallskip

{\em Case (i):} $i \in J$. This is identical to Cases (i) and (ii) in the proof of Theorem~\ref{the:odd}.

{\em Case (ii):}  $i \not\in J$. Observe that $\C_i$ has a cycle traversing some vertex in $S$, but none of the cycles in $\C_i$ traverse the edges of $E_S$. We construct $\C_i'$, a spanning cycle decomposition for $H_i^\ast$, as in Case (iii) of the proof of Theorem~\ref{the:odd}.

\smallskip

In each case, it can be verified that  $\G(\C_i')$ is connected, whence $\C_i'$ is a cycle decomposition of a spanning Euler tour of the  corresponding hypergraph. Hence (1) and (2) follow.

\medskip

$(\Leftarrow)$ Assume there exists an odd-size subset $J \subseteq \{ 1,\ldots,k \}$ such that (1) and (2) hold, and choose any $\ell \in J$. For each $i \in J$, let $T_i$ be a spanning Euler tour of $H_i'$, and for each $i \not\in J$, let $T_i$ be a spanning Euler tour of $H_i^\ast$. Let $\F_i=\{ T_i\}$ and proceed as in the proof of Theorem~\ref{the:odd} to construct a spanning cycle decomposition $\C$ of $H$.
Observe that, since $\G(\C_{\ell})$ contains both vertices in $S$ and is a connected subgraph of $\G(\C)$, the latter graph is connected and $\C$ is a cycle decomposition of an Euler tour of $H$.
\end{proof}

\section{Conclusion}

We showed that the presence of small vertex cuts allows for the reduction of the problem of existence of a spanning Euler family (tour) in a hypergraph $H$ to its derived hypergraphs. We would hereby like to propose that a similar reduction may be possible in the presence of particular edge cuts in $H$, which would extend a similar analysis for Euler families (tours) and cut edges initiated in \cite{eulerianHG, Quasi-eulerianHG}.

\bigskip

\centerline{\bf Acknowledgement}

\medskip

The authors gratefully acknowledge support by the Natural Sciences and Engineering Research Council of Canada (NSERC). Most of this work was completed during the second author's tenure of an NSERC Undergraduate Student Research Award.


\begin{thebibliography}{99}
\bibitem{Quasi-eulerianHG}
M. A. Bahmanian and M. \v{S}ajna,
Quasi-eulerian hypergraphs,
{\em Electron. J. Combin.} {\bf 24} (2017), \#P3.30, 12 pp.

\bibitem{eulerianHG}
M. A. Bahmanian and M. \v{S}ajna,
Eulerian properties of hypergraphs,
ArXiv:1608.01040 (2015).

\bibitem{HGconnection}
M. A. Bahmanian and M. \v{S}ajna,
Connection and separation in hypergraphs,
{\em  Theory Appl. Graphs} {\bf 2} (2015), no. 2, Art. 5, 24 pp.

\bibitem{GT} J. A. Bondy, U. S. R. Murty, {\em Graph theory}. Graduate Texts in Mathematics {\bf 244}, Springer, New York, 2008.

\bibitem{lonc-naroski} Z. Lonc, P. Naroski,  On tours that contain all edges of a hypergraph, {\em Electron. J. Combin.} {\bf 17} (2010), \# R144, 31 pp.

\bibitem{Lov} L. Lov\'{a}sz,  The factorization of graphs II,
{\em Acta Math. Acad. Sci. Hungar.} {\bf 23} (1972), 223--246.
\end{thebibliography}
\end{document}